\documentclass[12pt,reqno]{amsart}
\usepackage[T1]{fontenc}
\usepackage{amsfonts}
\usepackage{amssymb}
\usepackage{mathrsfs}
\usepackage[latin1]{inputenc}
\usepackage{amsmath}
\usepackage[english]{babel}

\newtheorem{theorem}{Theorem}[section]
\newtheorem{lemma}[theorem]{Lemma}

\newtheorem{remark}[theorem]{Remark}
\newtheorem{prop}[theorem]{Proposition}

\newtheorem{corollary}[theorem]{Corollary}

\numberwithin{equation}{section}

\newcommand{\D}{{\mathbb D}}

\newcommand{\C}{{\mathbb C}}
\newcommand{\N}{{\mathbb N}}

\newcommand{\cL}{{\mathcal L}}

\newcommand{\g}{\gamma}
\newcommand{\wh}{\widehat}
\newcommand{\vp}{\varphi}
\newcommand{\si}{\sigma}
\newcommand{\su}{\subseteq}

\newcommand{\la}{\langle}
\newcommand{\ra}{\rangle}
\newcommand{\wt}{\widetilde}

\newcommand{\sC}{\mathsf{C}}
\newcommand{\ov}{\overline}

\newcommand\Ker{\mathop{\rm Ker}}
\newcommand{\Bo}{\hfill $\Box$}

\newcommand{\abs}[1]{\left|#1\right|}
\newcommand{\op}{\operatorname}
\begin{document}

\title{The Ces\`aro operator in growth  Banach spaces of analytic functions}\footnote{The paper has been accepted for publication in Integral Equations and Operator Theory.}

\author{Angela\,A. Albanese, Jos\'e Bonet, Werner \,J. Ricker}



\address{ Angela A. Albanese\\
Dipartimento di Matematica e Fisica ``E. De Giorgi''\\
Universit\`a del Salento- C.P.193\\
I-73100 Lecce, Italy}
\email{angela.albanese@unisalento.it}

\address{Jos\'e Bonet \\
Instituto Universitario de Matem\'{a}tica Pura y Aplicada
IUMPA \\
Universitat Polit\`ecnica de
Val\`encia\\
E-46071 Valencia, Spain} \email{jbonet@mat.upv.es}

\address{Werner J.  Ricker \\
Math.-Geogr. Fakult\"{a}t \\
 Katholische Universit\"{a}t
Eichst\"att-Ingol\-stadt \\
D-85072 Eichst\"att, Germany}
\email{werner.ricker@ku.de}
\markboth{A.\,A. Albanese, J. Bonet, W. \,J. Ricker}%
{\MakeUppercase{ }}

\begin{abstract}
The Ces\`aro operator $\sC$,  when acting in the classical growth Banach spaces  $A^{-\g}$ and $A_0^{-\g}$, for $\g > 0 $,
of analytic functions on $\D$, is investigated. Based on a detailed knowledge of their spectra (due to A.\ Aleman and
A.-M.\ Persson) we are able to determine the norms of these operators precisely. It is then possible to characterize  the
mean ergodic and related properties of $\sC$ acting in these spaces. In addition, we determine the largest Banach space of
analytic functions on $\D$ which $\sC$ maps into $A^{-\g}$ (resp.\ into $A_0^{-\g}$); this \textit{optimal domain} space
always  contains $A^{-\g}$ (resp.\ $A_0^{-\g}$) as a \textit{proper} subspace.
\end{abstract}

\maketitle

%\newpage
\markboth{A.\,A. Albanese, J. Bonet and W.\,J. Ricker}%
{\MakeUppercase{The Ces\`aro operator in growth Banach spaces}}

\section{Introduction}

Let $H(\mathbb{D})$ denote the Fr\'echet space of all holomorphic  functions $f\colon \D\to \C$ equipped with the  topology of uniform
convergence on the compact subsets of the open unit disc  $\D$. The classical Ces\`aro operator $\sC$ is given by
\begin{equation}\label{eq.op-C}
f\mapsto \sC(f)\colon z\mapsto  \frac{1}{z}\int_0^z \frac{f(\zeta)}{1-\zeta} d \zeta, \ \ \ z \in \D\setminus\{0\}, \quad {\rm and }\quad \sC(f)(0)=f(0),
\end{equation}
for $f\in H(\D)$. It is a Fr\'echet space isomorphism of $H(\D)$ onto itself. In terms of the Taylor coefficients
$\wh{f}(n):= \frac{f^{(n)}(0)}{n!}$, for $ n \in \N_0,$ of functions $f(z)= \sum_{n=0}^\infty \wh{f}(n) z^{n} \in H(\D)$ one has the description
$$
\sC (f)(z) = \sum_{n=0}^{\infty} \Big( \frac{1}{n+1} \sum_{k=0}^{n} \wh{f}(k) \Big) z^n, \ \ \ \ z \in \D.
$$

A vector space $X \subseteq H (\D)$ is called a \textit{Banach space of analytic functions on $\D$} if it is a Banach space relative to a norm for which the natural inclusion of $X$ into $H(\D)$ is continuous. Since  evaluation at points of $\D$ are continuous linear functionals on $H(\D)$, this is equivalent to each evaluation functional $\delta_z : f \mapsto f (z)$
 at a point  $ z \in \D$ being an element of the dual Banach space $X^*$ of $X$. In this case, the Ces\`aro operator
 $\sC$ is said to \textit{act in} $X$ if it maps $X$ into itself, i.e.\ $\sC(X) \subseteq X$; by the closed graph theorem $\sC$ is necessarily continuous. It is known that $\sC$ acts in many classical Banach spaces of analytic functions on $\D$; for instance, the Hardy spaces $H^p(\D), 1 \le p<\infty,$ the Bergman
 and the  Dirichlet spaces, etc.\ (see \cite{A,AC,AS1,AS2,S} and the references therein).
 It also acts in certain weighted Banach spaces of analytic functions on $\D$, such as
 weighted Hardy spaces $H^p(w), 1<p<\infty,$  and others,  \cite{AMN, AP, CR, Pe}.
  On the other hand, $\sC $  \textit{fails} to act in
  $H^\infty (\D)$  since $\sC(\mathbf{1})(z)= (1/z) \log(1/(1-z))$, for $ z \in \D$; see \cite{DS} for an investigation of $\sC$ on $H^\infty (\D)$.
 Once $\sC$ is known to act in $X$, it is desirable to identify its spectrum $\sigma(\sC)$  and its point spectrum $\sigma_{pt}(\sC)$
as these  often influence various operator theoretic properties of $\sC$ (e.g.\ power boundedness, mean ergodicity, linear dynamics, decomposability, etc).

The aim of this note is to investigate various properties of the Ces\`aro operator $\sC$ related to its action in the classical
growth Banach space $A^{-\gamma}$,   for each  $\gamma >0$, consisting of  those analytic functions on $\D$ specified by
\[
A^{-\gamma}:=\{f\in H(\D)\colon \|f\|_{-\gamma}:=\sup_{z\in \D}(1-|z|)^\gamma|f(z)|<\infty\}
\]
and its (proper)  closed subspace
\[
A_0^{-\gamma}:=\{f\in H(\D)\colon \lim_{|z|\to 1^-}(1-|z|)^\gamma |f(z)|=0\ \},
\]
equipped with the norm
$$
\|f\|_{-\gamma}:= \sup_{z\in \D}(1-|z|)^\gamma|f(z)|, \ \ \ \ f \in A^{-\gamma}.
$$
Since $(1-|z|)^\gamma \le (1- |z|^2)^\gamma \le 2^\gamma (1-|z|)^\gamma$, for $z \in \D$, these are the same spaces
(with an equivalent norm) as those treated  in \cite[Ch.\,4]{HKZ}.

The space $A_0^{-\gamma}$ coincides with the closure of the polynomials in $A^{-\gamma}$,
 \cite[Lemma 3]{SW}, and point evaluations on $\D$ belong to both $(A^{-\gamma}_0)^* $ and $(A^{- \gamma})^*$,
 \cite[Lemma 1]{SW}. Moreover, the bidual $(A_0^{-\gamma})^{**} = A^{-\gamma}$ for all $\gamma > 0$, \cite{RS}, \cite[Theorem 2]{SW}.
 Shields and Williams, \cite{SW}, proved that $A^{-\gamma}$ is isomorphic to $\ell^\infty$ and that $A_0^{-\gamma}$ is isomorphic to $c_0$,
 for each $ \g > 0 $; see also Theorem~1.1 in the paper  by Lusky, \cite{L}. The growth Banach  spaces $A^{-\gamma}$ and $A_0^{-\gamma}$ play
an important  role in connection with the interpolation and sampling of holomorphic functions,
 \cite[Ch.\,4 \& 5]{HKZ}.  They  are particular examples of weighted Banach spaces $H^\infty_v$ and $H^0_v$ of holomorphic functions
on $\D $,  which have been investigated by several authors since the work of Shields and Williams, \cite{SW};
see, for example, \cite{ BBG, BBT, BS, L} and the references therein. Observe, for each pair $0<\mu_1<\mu_2$, that $A^{-\mu_1}\su A^{-\mu_2}_0$.

Recently, Aleman and Persson have made an extensive investigation of various properties (including the spectrum)
of generalized Cesàro operators acting in a large class of Banach spaces of analytic functions on $\D$, \cite{AP}, \cite{Pe}. In
particular, their results apply to the classical Cesàro  operator $\sC$ given by \eqref{eq.op-C} when acting in the growth spaces $A^{-\gamma}, A_0^{-\gamma}$
for $\gamma > 0 $. For this setting our additional results will complement and extend their work. For instance, they show
 that $\sC$ acts in both $A^{-\g}$ (where we denote it by $\sC_{\g}$) and in $A_0^{- \g}$
(where we denote it by $\sC_{\g,0}$). This follows from Theorem 4.1 of \cite{AP} with $g(z)=-\log(1-z)$, for
$ z \in \D$,  but no quantitative estimates of the operator norms $\|\sC_{\g}\|$ and $\|\sC_{\g,0}\|$ are given. We show
in Theorem~\ref{precisenorms},   that
\begin{equation}\label{norms}
\|\sC_{\g}^n\| = \|\sC_{\g,0}^n\| = 1, \ \ {\rm for} \ \ \g \geq 1, \ \ \mbox{and} \ \ \|\sC_{\g}^n\| = \|\sC_{\g,0}^n\| = 1/\g^n, \ \  \mbox{for} \ \ 0 < \g < 1,
\end{equation}
\noindent
for each $n \in \N $. Crucial for the proof of Theorem~\ref{precisenorms}  is knowing the precise nature of the spectra
$\sigma(\sC_{\g})$ and $\sigma(\sC_{\g,0})$, which have been completely determined by Aleman and Persson, \cite{AP,Pe}.
This knowledge of the spectra, combined with (\ref{norms}), makes it possible to determine the mean ergodic properties
 of $\sC_{\g}$ and $\sC_{\g,0}$. Namely, it turns out that  both $\sC_{\g}$ and $\sC_{\g,0}$ fail to be mean ergodic when $0 < \g \leq 1$ but, both operators are
(even uniformly) mean ergodic for all $\g >1$; see Theorem \ref{ergodic} below. Somewhat surprisingly, the ranges
$\rm{Im}(I-\sC_{\g}) := \{(I-\sC_\g) (f) : f \in A^{-\g}\}$, resp.\  $\rm{Im}(I-\sC_{\g,0})$, are \textit{closed} in $A^{- \g}$, resp.\ $A_0^{- \g}$,
for all positive $\g \neq 1$ but, not for $\g =1$; see Theorem \ref{ergodic}. None of the Ces\`aro operators $C_{\g}$, $C_{\g,0}$, for $\g > 0$,
are supercyclic and hence, they also fail to be hypercyclic (cf.\ Proposition \ref{hyper}).

The final section is devoted to determining the \textit{optimal domain} (a concept introduced in \cite{CR}) for both of the operators $\sC_{\g}$ and
$\sC_{\g,0}$ for all $\g >0$. That is, we are able to identify  the \textit{largest} Banach space of analytic functions  on $\D$, denoted by
$[\sC, A^{-\g}]$, (resp.\ by  $[\sC,A_0^{-\g}]$),  which $\sC$  maps continuously into
$A^{-\g}$ (resp.\ into $A_0^{-\g}$). More specifically, it is shown in Theorem \ref{largerextbig}   that
$$
[\sC,A^{-\g}] = \{ f \in H(\D) \ : \ f(z)/(1-z) \in A^{-(\g+1)} \}, \ \ \ \ \g > 0,
$$
equipped with the  norm $\|f\|=\| f(z)/(1-z) \|_{-(\g+1)}$. A similar description is available for $[\sC,A_0^{-\g}]$; see Theorem \ref{largerextsmall}.
An important feature is that the containment $A^{-\g} \subseteq [\sC,A^{-\g}]$ is \textit{proper}, that is,
$\sC_{\g}: A^{-\g} \rightarrow A^{-\g}$ has a \textit{genuine} extension $\sC: [\sC,A^{-\g}] \rightarrow A^{-\g}$
to its optimal domain space, as is the case for $\sC_{\g, 0 }$ acting in  $A_0^{-\g} $. As a consequence,
it is shown that the nested family of spaces $A^{-\g} \subsetneqq A^{-\beta}$ (resp.\ $A_0^{-\g} \subsetneqq A_0^{-\beta}$),
for $0 <\g < \beta$, has the property that $\sC$  fails to map \textit{every} larger space
$A^{-\beta}$ (resp.\   $A_0^{-\beta}$) into $A^{-\g}$ (resp.\ into $A_0^{-\g}$).
In particular, $A^{- \beta}\nsubseteq [\sC , A^{-\g}]$ and $A_0^{-\beta} \nsubseteq [\sC, A_0^{-\g}]$ for
every $\beta > \g$. On the other hand, the Banach space of analytic functions $ [\sC , A_0^{-\g}]$ is
\textit{non-comparable} with $A^{-\g}$ (i.e.,  $A^{-\g} \nsubseteq [\sC , A_0^{-\g}]$ and $ [\sC , A_0^{-\g}] \nsubseteq A^{-\g}$)
and, akin to $A^{-\g}$, is also a \textit{proper}  subspace of the optimal domain space $ [\sC , A^{-\g}]$;
see Proposition~\ref{p36}.

Our notation for concepts from functional analysis and operator theory is standard and we  refer  to \cite{DSI,ME,MV},
for example.
For Banach spaces of analytic functions see \cite{Du,HKZ}, for mean ergodic operators  \cite{DSI,K},
and for linear dynamics \cite{BM,GP}.

\section{Mean ergodic Properties of $\sC_\g$ and $\sC_{\g,0 }$}\label{meanergodic}

As noted above,  $A^{-\gamma}$ is canonically isomorphic to the bidual $(A_0^{-\g})^{**}$ of the Banach space
$A_0^{-\gamma}$.   In terms of this biduality the operator $\sC_{\g} \colon A^{-\gamma}\to A^{-\gamma}$
is the bidual operator $(\sC_{\g, 0})^{**}$ of   $\sC_{\g,0} \colon A_0^{-\gamma}\to A_0^{-\gamma}$.

We  recall the following result of Aleman and Persson (see \cite[Theorem 4.1]{Pe} and
\cite[Theorems 4.1, 5.1 and Corollaries 2.1, 5.1]{AP}, that will be used on  several occasions.

\begin{theorem}\label{T.Pe-AP}
Let $\g > 0 $. The Ces\`aro operator  $\sC_{\g,0}\colon A_0^{-\gamma}\to A_0^{-\gamma}$
 has the following properties.
\begin{itemize}
\item[\rm (i)] $\sigma_{pt}(\sC_{\g,0})=\{\frac{1}{m}\colon m\in\N,\ m<\gamma\}$.
\item[\rm (ii)] $\sigma(\sC_{\g,0})=\sigma_{pt}(\sC_{\g,0})\cup\left\{\lambda\in \C\colon \left|\lambda -\frac{1}{2\gamma}\right|\leq \frac{1}{2\gamma}\right\}$.
\item[\rm (iii)] If $\left|\lambda -\frac{1}{2\gamma}\right|< \frac{1}{2\gamma}$
{\rm(}equivalently ${\rm Re}\left(\frac{1}{\lambda}\right)>\gamma${\rm)}, then  ${\rm Im}(\lambda I-\sC_{\g,0})$
is a closed subspace of   $A_0^{-\gamma}$ and has codimension $1$.
\end{itemize}
Moreover, the Cesàro operator $\sC_\g : A^{-\g} \to A^{-\g}$ satisfies
\begin{itemize}
\item[\rm (iv)] $\sigma_{pt}(\sC_{\g})=\{\frac{1}{m}\colon m\in\N,\ m\leq \gamma\}$, and
\item[\rm (v)] $\sigma(\sC_{\g})=\sigma(\sC_{\g,0})$.
\end{itemize}
\end{theorem}

As an immediate consequence we have the following result.

\begin{corollary}
For  $\g > 0 $, neither of the operators $\sC_\g, \sC_{\g, 0}$ is weakly compact.
\end{corollary}

\begin{proof}
Fix $\g > 0$. As noted in Section 1 there exists a Banach space isomorphism $\Phi$ from $c_0$ onto
$A_0^{-\g}$. Suppose that $\sC_{\g,0}$ is weakly compact. Then $\wt{\sC} := \Phi^{-1} \sC_{\g, 0} \Phi $ is
weakly compact from $c_0$ into $c_0$ and hence, $\wt{\sC}$ is also compact, \cite[Ex.\,3.54(b), p.347]{ME}.
Accordingly, $\sC_{\g,0} = \Phi \, \wt{\sC} \Phi^{-1} $ is compact. But, this is impossible as $\si (\sC_{\g,0})$
is an uncountable set; see Theorem~\ref{T.Pe-AP}(ii). Hence, $\sC_{\g,0}$ cannot be  weakly  compact.

By Gantmacher's Theorem, \cite[Ch.\,VI, Sect.\,4, Theorem 8]{DSI}, also $\sC_\g = (\sC_{\g,0})^{**}$ cannot be weakly
compact.
\end{proof}

The continuity of $\sC_{\g}$ and $\sC_{\g,0}$ as established in \cite[Theorem 4.1]{AP} gives no quantitative estimate for
their operator  norm.
So, the following result is of some interest.

\begin{theorem}\label{precisenorms}
%\begin{enumerate}
%\item[(i)]
{\rm(i)} Let $\g \geq 1$. Then $\|\sC_{\g}^n\| = \|\sC_{\g,0}^n\| = 1$ for all $n \in \N  $.

%\item[(ii)]
{\rm(ii)} Let $0<\g < 1$. Then $\|\sC_{\g}^n\| = \|\sC_{\g,0}^n\| = 1/\g^n$ for all $n \in \N$.
%\end{enumerate}
\end{theorem}

\begin{proof}

Let $\g >0$. Fix $f\in A^{-\gamma}$ with $\abs{\abs{f}}_{\gamma} \leq 1$. Then
 $\abs{f(\xi)} \leq (1-\abs{\xi})^{-\gamma} $ for $ \xi \in \mathbb{D}$.
 For $z\in\mathbb{D} \setminus \{0\}$,  it follows from \eqref{eq.op-C}  and the previous inequality for $|f|$ that
\begin{eqnarray*}
 \abs{\sC_\g (f) (z)} & = &  \frac{1}{\abs{z}} \abs{\int_0^1 \frac{f(tz)}{1-tz}z \: dt} \le
  \int_0^1 \frac{\abs{f(tz)}}{\abs{1-tz}} dt \leq \int_0^1 \frac{\abs{f(tz)}}{1-t\abs{z}} dt \\
& \leq & \int_0^1 \frac{dt}{(1-t\abs{z})^{\gamma + 1}} = \left[ \frac{1}{\gamma\abs{z}}(1-t\abs{z})^{-\gamma}\right]_{t=0}^{t=1} \\ & = & \frac{1}{\gamma\abs{z}} \left( \frac{1}{(1-\abs{z})^{\gamma}} - 1\right) = \frac{1}{(1-\abs{z})^{\gamma}}\frac{1-(1-\abs{z})^{\gamma}}{\gamma \abs{z}}.
\end{eqnarray*}
The non-negative function $\phi(s):= \frac{1-(1-s)^{\gamma}}{s}$ for $s \in (0,1]$ and $\phi(0) := \g$  is
 continuous as $\lim\limits_{s\rightarrow 0^{+}} \frac{1-(1-s)^{\gamma}}{s} = \gamma$.
 Set $M_{\gamma} := \underset{s\in[0,1]}{\op{sup}} \phi (s) $. Since $ |\sC_\g  (f)(0) |= |f(0)| \le 1 $,
 we have $1\leq \underset{z\in\mathbb{D}}{\op{sup}}(1-\abs{z})^{\gamma}\abs{\sC_\g (f)(z)} \leq \max \{1,\frac{M_{\gamma}}{\gamma}\}$. Hence,

\begin{equation}\label{eq.max}
1 \le \|\sC_\g\| \le \max \Big \{1, \frac {M_\g}{\g} \Big \} , \qquad \g > 0 .
\end{equation}

(i) Let now $\gamma \geq 1$. The claim is  that $\frac{M_{\gamma}}{\gamma} \leq 1$,  which then  implies that
$\abs{\abs{\sC_{\g}}} = 1$. To verify the claim it suffices to show that the function
$\psi(s) := (1-s)^{\gamma}-1+\gamma s$,  for $s\in [0,1]$, is non-negative. Observe that  $\psi(0) = 0$ and
 $\psi(1)=(\gamma - 1) \geq 0$. If $\gamma =1$, then  $\psi(s) =0 $ for all $ s\in[0,1]$ and we are done.
 In case $\gamma > 1$ we have  $\psi'(s) = -\gamma(1-s)^{\gamma - 1} + \gamma = \gamma\left[1-(1-s)^{\gamma - 1}\right] \ge  0$
 and so $\psi$ is increasing. In particular, $\psi (s) \ge \psi (0) = 0$ for $s \in [0,1]$, as required.

Clearly, $\|\sC_\g\| \le 1 $ implies that  $\|\sC_{\g}^n\| \le \|\sC_\g\|^n \leq 1$ for  $n \in \N$. On the other hand, Theorem \ref{T.Pe-AP}
shows that $1 \in \sigma(C_{\g})$ and so, by the spectral mapping theorem, $1 \in \sigma(C^n_{\g})$.
In particular, the spectral radius $r(C^n_{\g}) \geq 1$,  for  $n \in \N$, and so  $1 \leq r (\sC_\g^n) \le \|\sC_{\g}^n\|$ for
$n \in \N$. This completes the proof of (i) for $\sC_{\g}$. The statement for $\sC_{\g,0}$ now follows from
$\|\sC^n_{\g, 0}\| = \|(C^n_{\g, 0})^{**}\| = \|\sC_\g^n\|$ for $n \in \N $.

(ii) Let $0<\g < 1$. It follows from Theorem \ref{T.Pe-AP} (as $\sigma_{pt}(C_{\g,0})=\emptyset$ in this case),
that $\sigma(C_{\g,0}) = \left\{\lambda\in \C\colon \left|\lambda -\frac{1}{2\gamma}\right|\leq \frac{1}{2\gamma}\right\}$.
In particular, $[0, \frac1 \g] \subseteq \sigma(C_{\g,0})$. By the spectral mapping theorem $[0, \frac {1}{ \g^n}] \subseteq \sigma(C_{\g,0}^n)$
for each $n \in \N$, and so $1/\g^n \leq r(C_{\g,0}^n) \leq \|C_{\g,0}^n\| = \|C_{\g}^n\|$ for each $n \in \N$.

To establish the reverse inequality we show, because   of  $0<\g < 1$, that $M_{\g} \leq 1$. Indeed, for
 $0<s<1$ we have  $(1-s) < (1-s)^{\g}$, i.e., $1-(1-s)^{\g} < s$, and hence,  $0<\phi(s)<1$. Since $\phi(0)=\g <1$ and
 $\phi(1)=1$, it is clear that  $M_{\g} \leq 1$ and so $\frac{M_{\gamma}}{\gamma} \le \frac{1}{\g}$. It follows from \eqref{eq.max}
 that  $\|C_{\g}\| \leq  \frac{ 1}{ \g}$. Hence, $\|\sC_{\g}^n\| = \|\sC_{\g,0}^n\| \leq  \frac{ 1}{ \g^n}$ for all $n \in \N$, and the proof is complete.
\end{proof}

For a Banach space $X$,  denote by $\cL(X)$ the space of all bounded linear operators from $X$ into  itself.
Recall that an operator $T\in\cL(X)$ is  \textit{mean ergodic} if its sequence of Ces\`aro averages
\begin{equation}\label{eq.cesaro-aver}
T_{[n]}:=\frac{1}{n}\sum_{m=1}^n T^m, \quad n\in\N,
\end{equation}
 converges to some operator $P\in \cL(X)$ in the strong operator topology $\tau_s$, i.e., $\lim_{n\to\infty}T_{[n]}x=Px$ for each $x\in X$, \cite[Ch.VIII]{DSI}.
 It follows from \eqref{eq.cesaro-aver} that $\frac{T^n}{n}=T_{[n]}-\frac{n-1}{n}T_{[n-1]}$, for $n\geq 2$. Hence,  $\tau_s$-$\lim_{n\to\infty}\frac{T^n}{n}=0$ whenever $T$ is mean ergodic and, in particular, $\sup_{n} \frac{\|T^n\|}{n} < \infty$.
 According to \cite[VIII Corollary 5.2, p.662]{DSI}, when $T$ is mean ergodic one  has the direct decomposition
\begin{equation}\label{eq.decompo}
X=\Ker (I-T)\oplus \ov{\mbox{Im}(I-T) }  .
\end{equation}

A Banach space operator $T\in \cL(X)$ is called \textit{uniformly mean ergodic} if there exists $P\in \cL(X)$ such that $\lim_{n\to\infty}\|T_{[n]}-P\|=0$.
It is then  immediate that necessarily $\lim_{n\to\infty}\frac{\|T^n\|}{n}=0$. A  result of M.Lin, \cite{Li}, states that a Banach space operator $T\in \cL(X)$ satisfying $\lim_{n\to\infty}\frac{\|T^n\|}{n}=0$ is uniformly mean ergodic if and only if $\mbox{Im}(I-T)$ is a \textit{closed} subspace of $X$.

An operator $T\in \cL(X)$ is called \textit{power bounded} if $\sup_{n\in\N}\|T^n\|<\infty$. Since $T_{[n]} (I-T)= \frac 1 n (T- T^{n+1})$
for $n \in \N$, it follows that
\begin{equation}\label{eq.lim}
    \lim_{n \to \infty } T_{[n]} x = 0 , \qquad x \in \mbox{Im} (I-T),
\end{equation}
whenever $T$ is power bounded.

\begin{lemma}\label{lemmaergodic}
%\begin{itemize}
%\item[(i)]
{\rm(i)} Define $\vp \in H (\D)$ via $\vp (z):= 1 / (1-z)$ for $z \in \D $. The Ces\`aro operator $\sC: H(\D) \rightarrow H(\D)$ satisfies
$\Ker(I-\sC)= {\rm span} \{ \vp\}$ and
$$
{\rm Im}(I-\sC)= \{h \in H(\D) : h(0)=0 \} = \Ker (\delta_0) ,
$$
with $\delta_0 \in H (\D)^*$. In particular, ${\rm Im} (I-C)$ is closed in $H (\D)$.

%\item[(ii)]
{\rm(ii)} Let $X$ be any Banach space of analytic functions on $\D$  which contains the constant functions, is  continuously
included in $H(\D)$ and such that $\sC: X \rightarrow X$ is continuous.     If $\sC: X \rightarrow X$ is mean ergodic,
then  $ \vp \in X$ and $\Ker (I-C) = {\rm span} \{\vp\}$.
%\end{itemize}
\end{lemma}

\begin{proof}
(i) This is proved  in \cite[Section 2]{Pe}; see (2.4) and (2.5) on page 1184 with $\lambda =1$.

(ii) Assume that $\sC: X \rightarrow X$ is continuous and mean ergodic. Let $\sC_X $ denote the restriction of
$\sC$ to $X$.   Then
 \begin{equation}\label{eq.im}
     X=\Ker(I-\sC_X) \oplus \ov{{\rm Im}(I- \sC_X)} .
 \end{equation}
 If $\vp  \notin X$, then $\Ker (I - \sC_X) = \{0\}$ by part (i). Since point evaluations on $\D$ belong to $X^*$ (see
 Section 1), it follows from \eqref{eq.im} that
 $X \subseteq \{h \in H(\D) : h(0)=0 \}$. This is a contradiction, since the constant function $\textbf{1}$ belongs to $X$.
Hence, $\vp \in X$. Since $\Ker (I- \sC_X) \subseteq \Ker (I- \sC)$, it follows from part (i) that $\Ker (I- \sC_X) =
\mbox{span} \{\vp\}$.
\end{proof}

We can now establish the main result of this section.

\begin{theorem}\label{ergodic}
%\begin{itemize}
%\item[(i)]
{\rm (i)}   Let $0<\g < 1$.
Both of the  operators  $\sC_{\g}$ and $\sC_{\g,0}$ fail to be power bounded and are not mean ergodic.
Moreover,
$$
\Ker(I- \sC_{\g}) = \Ker(I - \sC_{\g,0}) = \{ 0 \} ,
$$
and ${\rm Im}(I- \sC_{\g})$  {\rm(}resp.\ ${\rm Im}(I - \sC_{\g,0})${\rm)}
is a proper closed subspace of $A^{-\g}$ {\rm(}resp.\ of  $A_0^{-\g}${\rm)}.

%\item[(ii)]
{\rm (ii)}  Both of the operators  $\sC_{1}$ and $\sC_{1,0}$ are power bounded but not mean ergodic.
Moreover, ${\rm Im}(I- \sC_{1})$ {\rm(}resp.\ ${\rm Im}(I - \sC_{1,0})${\rm)} is not a closed subspace of
$A^{-\g}$ {\rm(}resp.\ of  $A_0^{-\g}${\rm)}.

%\item[(iii)]
{\rm (iii)} Let $\g > 1$. Both of the  operators $\sC_{\g}$ and $\sC_{\g,0}$ are power bounded and uniformly mean ergodic.
Moreover, ${\rm Im}(I- \sC_{\g})$ {\rm(}resp.\ ${\rm Im}(I - \sC_{\g,0})${\rm)} is a proper closed subspace of
$A^{-\g}$ {\rm(}resp.\ of $A_0^{-\g}${\rm)}. In addition,
\begin{equation}\label{eq.ho}
{\rm Im}(I- \sC_{\g}) = \{h \in A^{-\g} : h (0)= 0 \}    .
\end{equation}

Moreover, with  $\vp (z) := 1 / (1-z)$, for $z \in \D$, the linear projection operator $P_\g :A^{-\g}  \to A^{-\g}$ given by
$$
P_\g (f) := f (0) \vp, \qquad f \in A^{-\g} ,
$$
is continuous and satisfies $\lim_{n \to \infty } (\sC_{\g})_{[n]}= P_\g$ in the  operator norm.
%\end{itemize}
\end{theorem}

\begin{proof}
Clearly the function  $\varphi \in A^{-\g}_0 \subseteq A^{-\g}$ if $\g > 1$. Moreover,
 $\varphi \in A^{-1} \setminus A_0^{-1}$ and $\varphi \notin A^{-\g}$ in case $0<\g<1$.

(i) For $0<\g < 1$ it is clear from Theorem \ref{precisenorms}(ii) that both $\sC_{\g}$ and $\sC_{\g,0}$ are not
 power bounded. Neither are they  mean ergodic  because $\sup_{n} \|\sC_{\g}^n\|/n = \sup_{n} \|\sC_{\g,0}^n\|/n = \infty$;
see the discussion after \eqref{eq.cesaro-aver}.
Since $\Ker(I-\sC) = \rm{span}\{\varphi \}$ in $H(\D)$,  by Lemma \ref{lemmaergodic}(i) and because
$\vp \notin A^{-\g} $,  we can conclude that  $\Ker(I- \sC_{\g}) = \Ker(I - \sC_{\g,0}) = \{ 0 \}$.

Since $|1-\frac{1}{2\g}| < \frac{1}{2\g}$, we can apply Theorem \ref{T.Pe-AP}(iii) to deduce that ${\rm Im}(I - \sC_{\g,0})$
is a proper closed subspace of $A_0^{-\g}$. Moreover, $(I- \sC_{\g}) \in \cL (A^{-\g})$ is the bidual operator  of $I - \sC_{\g,0}$
 and so it follows from \cite[Proposition 2.1]{ABR-7} that ${\rm Im}(I- \sC_{\g})$ is also closed in $A^{-\g}$.
It is a proper subspace of $A^{-\g}$ because every $h \in {\rm Im}(I- \sC_{\g})$ satisfies $\sC_ \g (h)(0)=0$; see \eqref{eq.op-C}.

(ii) Both $\sC_{1}$ and $\sC_{1,0}$ are power bounded by Theorem \ref{precisenorms}(i).
Since $\varphi \notin A_0^{-1}$, it follows from Lemma \ref{lemmaergodic}(ii), with $X= A_0^{-1}$, that $\sC_{1,0}$
is not mean ergodic in $A_0^{-1}$. By \cite[Proposition 2.2]{ABR-7},  $\sC_{1}$  also fails to be  mean ergodic.

Assume  that ${\rm Im}(I- \sC_{1})$ is closed in $A^{-1}$. Since $ \lim_{n \to \infty }\|\sC_{1}^n\|/n = 0$,
 it follows from Lin's theorem, \cite{Li}, that  $\sC_{1}$ is uniformly mean ergodic.
But, we just argued that $\sC_1$  is not even  mean ergodic; contradiction!
The argument for $\sC_{1,0}$ is the same.

(iii) Let $\g > 1$. By Theorem \ref{precisenorms}(i) both $\sC_{\g}$ and $\sC_{\g,0}$ are power bounded.
In order to verify \eqref{eq.ho}, it is clear that
${\rm Im}(I- \sC_{\g}) \subseteq \{ h \in A^{-\g} \ : \ h(0)=0 \}$.

For the reverse inclusion,
fix $h \in A^{-\g}$ satisfying  $h(0)=0$. Theorem~\ref{T.Pe-AP} reveals that $1 \in \rho (\sC_\g)$.
By the calculations on p.1184 of \cite{Pe}, with $\lambda = 1 $,
 there is a unique $f \in H(\D)$ satisfying  $f(0)=0$ and $(I-\sC)f=h$. In fact, $f$ is given by \cite[formula (2.7)]{Pe}:
$$
f(z) = h(z) + \frac{1}{1-z} \int_0^z \frac{h(\zeta)}{\zeta}, \qquad z \in \D.
$$
In order to conclude that  $f \in A^{-\g}$, we first show that the function $H(z):=h(z)/z$, for $ z \in \D \setminus \{0\}$
and $H(0): =h'(0)$, belongs to  $ A^{-\g}$. Recalling  that $h (0) = 0 $, it is clear
that $H \in H (\D)$. If $\frac 1 2 \le |z| < 1 $, then $|H(z)| \leq 2 |h(z)|$.
Hence,  $(1-|z|)^{\g} |H(z)| \leq 2 \|h \|_{-\g}$. On the other hand, for  $|z| \leq \frac 1 2 $ we can apply
the maximum modulus principle to $H \in H(\D)$ to conclude that
$$
|H(z)| \leq \max_{|\zeta|=1/2} |H(\zeta)| = 2 \max_{|\zeta|=1/2} |h(\zeta)| = 2^{\g+1} \max_{|\zeta|=1/2}
(1- |\zeta|)^\g |h  (\zeta) | \le   2^{\g+1} \| h \|_{-\g}.
$$
Summarizing, $\|H\|_{- \g}= \sup_{z \in \D} (1-|z|)^{\g} |H(z)| \leq 2^{\g+1} \| h \|_{-\g}$, i.e.,  $H \in A^{-\g}$.

Consider  $G(z):= \int_0^z H(\zeta )\; d \zeta $, for $ z \in \D$. Clearly $G'(z)=H(z)$ for  $z \in \D$, with
 $H \in A^{-\g}$. By a classical result of Hardy and Littlewood (consider $p= \infty $ and $\beta := \g$ in
  \cite[Theorem 5.5]{Du}), the function  $G \in A^{-(\g-1)} = A^{-\g+1}$ (recall $(\g -1) >0)$). Therefore, for each $z \in \D$,
we have for $\vp G $   that
\begin{eqnarray*}
&& (1- |z|)^{\g} |\vp (z) G (z)| = (1-|z|)^{\g} \Big|\frac{1}{1-z} \int_0^z  H (\zeta) \; d \zeta \Big| \\
&&  \leq (1-|z|)^{\g-1}  \Big|\int_0^z  H (\zeta ) \; d \zeta\Big| \leq \| G \|_{-\g+1},
\end{eqnarray*}
that is, $\vp G \in A^{-\g}$. Hence, also $f = h + \vp G \in A^{-\g} $. This establishes that
 ${\rm Im}(I- \sC_{\g}))$ is closed in $A^{-\g}$. Since $\sC_{\g}$ is power bounded, a theorem of Lin \cite{Li}
 implies that   $\sC_{\g}$ is uniformly mean ergodic.

By \cite[Proposition 2.1]{ABR-7} it follows that also ${\rm Im}(I- \sC_{\g,0}))$ is closed in $A_0^{-\g}$. So,
$\sC_{\g,0}$ is uniformly mean ergodic, again as a consequence of Lin's theorem.

 Lemma \ref{lemmaergodic}(ii), with $X = A^{-\g}$, and the mean ergodicity of $\sC_\g$  yield that $\vp \in A^{-\g}$.
Since $\delta_0 \in (A^{-\g})^*$,     the formula
$$
P_\g (f) := f (0) \vp = \langle f, \delta_0\rangle \vp, \qquad f \in A^{-\g} ,
$$
implies that $P_\g $ is a continuous projection in $A^{-\g}$.

For each $ f \in A^{-\g}$ we have
\begin{equation}\label{eq.cg}
(\sC_{\g})_{[n]} (f) = (\sC_{\g})_{[n]} (f - f(0)) + f(0) (\sC_{\g})_{[n]} (\textbf{1})    , \qquad n \in \N .
\end{equation}
On the other hand, $\vp \in A^{-\g} $ satisfies $(\sC_{\g})_{[n]} (\vp) = \vp $ for all $n \in \N$; see Lemma \ref{lemmaergodic}(ii)
with $X = A^{-\g}$. According to \eqref{eq.ho} the function $h (z) := z \vp (z)$ for $z \in \D$ belongs to
${\rm Im}(I- \sC_{\g})$ and hence, the sequence
$$
(\sC_{\g})_{[n]} (\textbf{1}) = (\sC_{\g})_{[n]}  (\vp - h ) = \vp - (\sC_{\g})_{[n]} (h) , \qquad n \in \N ,
$$
converges to $\vp$ in $A^{-\g}$ for $n \to \infty $; see \eqref{eq.lim}. It then follows from \eqref{eq.cg} and
the fact that $\lim_{n \rightarrow \infty} (\sC_{\g})_{[n]} (f - f(0)) = 0 $ in $A^{-\g}$ via \eqref{eq.lim} (since
$(f-f(0)) \in  {\rm Im}(I- \sC_{\g})$ by \eqref{eq.ho}) that
$$
\lim_{n \to \infty } (\sC_{\g})_{[n]} (f) = f (0) \vp = P_\g (f) .
$$
This establishes that $\tau_s \mbox{-} \lim_{n \to \infty }  (\sC_{\g})_{[n]} = P_\g $. But, the uniform mean ergodicity of
$\sC_{\g}$ means that $\{(\sC_{\g})_{[n]} \}_{n \in \N }$ is a convergent sequence for the operator norm and hence,
its operator norm limit must  also be $P_\g$.
\end{proof}

\begin{remark}\label{remergodic} \rm{
(i) Let $\g \ge 1 $. It follows from Theorem \ref{T.Pe-AP}  that the boundary $\partial \D$ of $\D$ satisfies
$ \si (\sC_{\g}) \cap \partial \D = \{1\} =  \si (\sC_{\g, 0}) \cap \partial \D $.
Since both $\sC_{\g}$ and $\sC_{\g,0}$ are power bounded, it follows from  Theorem 1 and the Remark on p.317 of \cite{K-T}  that
$$
\lim_{n \to \infty}  \left\|\sC_{\g}^{n +1} - \sC_{\g}^n \right\| = 0 =\lim_{n \to \infty}  \left\|\sC_{\g,0}^{n +1} -  \sC_{\g,0}^n \right\|.
$$
(ii) Let $\g > 1 $.  Theorem \ref{T.Pe-AP} shows that $\lambda = 1$ is an isolated singularity of the resolvent map of both
$\sC_\g $ and $\sC_{\g,0}$. Since both $ \sC_{\g}$ and $ \sC_{\g,0}$ are uniformly mean ergodic, it follows that $1$ is actually a
simple pole of the resolvent  map, \cite[Theorem 2.7, p.90]{K}.

(iii) For $\g \ge 1$ the point $1 \in \si (\sC_\g) = \si (\sC_{\g,0})$ and so the spectral mapping theorem applied to the
polynomial $p_n (z) := \frac 1 n \sum^n_{m=1} z^m$ (i.e., $(\sC_\g)_{[n]} = p_n (\sC_\g)$) yields that
$1 \in \si ( (\sC_\g)_{[n]})$, for $n \in \N $. Hence,
$$
1 \le r \left( (\sC_\g)_{[n]}\right) \le \| (\sC_\g)_{[n]}\| , \qquad n \in \N .
$$
Then  Theorem \ref{precisenorms}(i) and the formula $ (\sC_\g)_{[n]} = \frac 1 n \sum^n_{m=1} \sC_\g^m$  imply that
$$
\left\| (\sC_\g)_{[n]}\right\| = \left\| (\sC_{\g,0})_{[n]}\right\| = 1 , \qquad n \in \N  .
$$

For $0 < \g < 1 $ the point $\frac 1 \g \in \si  (\sC_\g)$ and so again  the spectral mapping theorem yields that $ p_n (\frac 1 \g ) =
\frac 1 n \sum^n_{m=1}   \frac{1}{\g^m} \in \si ( (\sC_\g)_{[n]})$ for $n \in \N$. In particular,
$$
\frac 1 n \sum^n_{m=1} \frac{1}{\g^m} \le r \left( (\sC_\g)_{[n]}\right) = r\left ( (\sC_{\g,0})_{[n]}\right), \qquad n \in \N .
$$
It then follows from Theorem~\ref{precisenorms}(ii) that
$$
\left\| (\sC_{\g,0})_{[n]}\right\| =\left\| (\sC_\g)_{[n]}\right\| = \frac 1 n \sum^n_{m=1} \frac{1}{\g^m} =
\frac{(\frac{1}{\g^n})-1}{(n+1)(1-\g)} , \qquad n \in \N .
$$
In particular, for $0<\g < 1 $, not only do $\sC_\g $ and $\sC_{\g, 0}$ fail to be power bounded but, their Cesàro averages
$\{(\sC_\g)_{[n]}\}^\infty_{n=1}$ and $\{(\sC_{\g, 0})_{[n]}\}^\infty_{n=1}$ are also unbounded sequences in $\cL (A^{-\g})$ and
$\cL (A_0^{-\g})$, respectively.
\Bo
}
\end{remark}

Concerning the dynamics of $\sC$ recall that an operator $T\in \cL(X)$, with  $X$ a separable
Fréchet  space, is called \textit{hypercyclic} if there exists $x\in X$ such that the orbit $\{T^nx\colon n\in\N_0\}$
is dense in $X$. If, for some $z\in X$, the projective orbit $\{\lambda T^n z\colon \lambda\in\C,\ n\in\N_0 \}$ is dense
in $X$, then $T$ is called \textit{supercyclic}. Clearly, hypercyclicity  implies supercyclicity.
It is proved in \cite{ABR_power} that $\sC$ is not supercyclic on $H(\D)$. Since the image of a dense subset of
$A_0^{-\g}$ under the natural inclusion map into $H (\D)$ is dense in $H (\D)$, we  have the following  consequence.

\begin{prop}\label{hyper}
The Cesàro operator $\sC_{\g, 0 }$ is not supercyclic and  hence, also  not hypercyclic,  in each space $A_0^{-\gamma}$,  for  $ \g > 0$.
\end{prop}

\section{Optimal extension of $\sC_\g$ and  of   $\sC_{\g, 0}$}\label{optimal}

The \textit{optimal domain} $[\sC, H^p (\D)]$ of the Ces\`aro operator $C: H^p (\D) \to H^p (\D)$, $ 1 \le p < \infty $, was
introduced and thoroughly investigated in \cite[Section 3]{CR}. The definition  given there is rather general and can be applied to $\sC$
when it acts in any Banach space of analytic functions $X$ on $\D$. Namely,
$$
[\sC, X] := \{ f \in H(\D) \ : \ \sC(f) \in X \},
$$
which is a Banach space for the norm
\begin{equation}\label{eq.norm1}
  \|f\|_{[\sC , X]} := \|\sC (f)\|_X, \qquad f \in [\sC, X] ,
\end{equation}
as of a consequence of $\sC: H (\D) \to H (\D)$ being a topological Fréchet space isomorphism. When point
evaluations on $\D$ belong to $[\sC, X]^*$, then $[\sC, X] $ is actually a Banach space of analytic functions on $\D$
and $\sC$ maps $[\sC, X]$ isometrically onto $X$. If,  in addition, $\sC $ acts in $X$, then $X \su [\sC, X]$ and the
natural inclusion map is continuous. Most important is that $[\sC, X]$ is the \textit{largest} of all Banach spaces of
analytic functions $Y$ on $\D$ that $\sC$ maps continuously into $X$; the argument is analogous to \cite[Remark 3.1]{CR}.
To describe more concretely which functions from $ H (\D)$ are members of $[\sC, X]$ may not be easy in general:
for $X= H^p (\D)$ this is based on properties of the Littlewood-Paley $g$-function \cite[Proposition 3.2 \& Corollary 3.3 ]{CR}.
To express the norm \eqref{eq.norm1} in a more explicit way, if possible, would also be an advantage.

The aim of this section it to investigate the optimal domain spaces $[\sC, A^{-\g}]$ and $ [\sC, A_0^{-\g}]$ for $\g > 0$. Clearly
$[\sC, A_0^{-\g}] \su [\sC, A^{-\g}]$. Moreover, with continuous inclusions, we clearly have
$$
A^{-\g}\su [\sC, A^{-\g}]  \subsetneqq H (\D)  \quad \mbox{and} \quad  A_0^{-\g} \su [\sC, A_0^{-\g}] \subsetneqq H (\D) ,
$$
due to the closed graph theorem and the fact that point evaluations on $\D$ belong to $[\sC, A^{-\g}]^*$ hence, also to
$[\sC, A_0^{-\g}]^* $.
To verify this latter claim, fix $z_0 \in \D$ and choose $0<r<1$ with $|z_0| < r$. For $f \in  H(\D)$
observe that  $f(z) = (1-z)(z \sC(f)(z))'$, for $ z \in \D,$ which is routine to verify. Given $f \in [\sC , A^{-\g}]$,
the previous identity  and   the Cauchy integral formula  yield
$$
|\la f, \delta_{z_0} \ra| = |f(z_0)| = |1-z_0| \cdot \Big|\frac{1}{2 \pi i} \int_{|\zeta|=r} \frac{\zeta \sC(f)(\zeta) }{(\zeta-z_0)^2} d \zeta \Big| \leq
$$
$$
\frac{|1-z_0|r^2}{(r-|z_0|)^2(1-r)^{\g}} \sup_{|\zeta|=r}(1-|\zeta|)^{\g} |\sC(f)(\zeta)| \leq
\frac{|1-z_0|r^2}{(r-|z_0|)^2(1-r)^{\g}} \|\sC(f) \|_{-\g}.
$$
Accordingly, $|f(z_0)| \leq \frac{|1-z_0|r^2}{(r-|z_0|)^2(1-r)^{\g}} \| f \|_{[\sC, A^{-\g}]}$ and so  each  evaluation functional
$\delta_z : f \rightarrow f(z)$ is continuous on both $[\sC, A^{-\g}]$ and $[\sC, A_0^{-\g}]$, for each $z \in \D$.

\begin{theorem} \label{largerextbig}
Let $\g > 0$ and $\vp (z) := 1 / (1-z)$ for $z \in \D$. The optimal domain $[\sC, A^{-\g}]$ of $\sC_{\g}: A^{-\g} \rightarrow A^{-\g}$
is isometrically isomorphic to $A^{-\g}$ and is given by
\begin{equation}\label{eq.hd}
[\sC, A^{-\g}]  =  \{ f \in H(\D) \ : \ f  \varphi \in A^{-(\g+1)} \}  .
\end{equation}
Moreover, the norm $\| . \|_{[\sC, A^{-\g}]}$ is equivalent to the norm $f \rightarrow \|f \varphi \|_{-(\g+1)}$
and the containment $A^{-\g} \su [\sC, A^{-\g}]$ is proper.
\end{theorem}
\begin{proof}
That $[\sC, A^{-\g}]$ is isometrically isomorphic to $A^{-\g}$ follows from the facts that
 $\sC: [\sC, A^{-\g}] \rightarrow A^{-\g}$ is both injective (as $\sC$ is bijective on $H(\D)$)
and surjective (by the definition of $[\sC, A^{-\g}]$), and that it is an isometry (via \eqref{eq.norm1}).

To verify \eqref{eq.hd}  observe,  for $f \in H(\D)$, that   $\sC(f) \in A^{-\g}$ if and only if the function
$z \mapsto \int_0^z f(\zeta)/(1-\zeta) d \zeta \in A^{-\g}$. This follows from  the argument in the proof of Theorem \ref{ergodic}(iii)
above, where it is shown that if $h \in A^{-\g}$ satisfies $h (0) =0$, then also $H (z) :=  \frac{h (z)}{z} \in A^{-\g}$ (the
converse is clear from $h (z) = z H (z)$). Note that this part of the argument in the proof of Theorem \ref{ergodic}(iii) did
not require $\g > 1 $ (which was being assumed there).
By a result of Hardy and Littlewood, \cite[Theorem 5.5]{Du}, this in turn is equivalent to $f \vp \in A^{-(\g+1)}$.

To verify the equivalence of the norms $\| . \|_{[\sC, A^{-\g}]}$ and $f \rightarrow \|f \varphi \|_{-(\g+1)}$ we proceed as follows.
 First,   the space $E:=\{ f \in H(\D) \ : \ f \varphi \in A^{-(\g+1)} \}$ endowed with the norm $\|f \varphi \|_{-(\g+1)}$,
 for $  f \in E, $ is a Banach space which is  isomorphic to $A^{-(\g+1)}$. On the other hand, since $\sC : [\sC, A^{-\g}] \rightarrow A^{-\g}$
is continuous  we have, for some constant $K > 0$, that
$$
\|f \varphi\|_{-(\g+1)} \leq \| f \|_{-\g} \leq K \|f\|_{[\sC, A^{-\g}]}, \qquad f \in [\sC, A^{-\g}].
$$
Therefore the identity map $I: [\sC, A^{-\g}] \rightarrow E$ is a continuous bijection between Banach spaces.
By the open mapping theorem its inverse is also continuous. This implies that the two norms are equivalent in $[\sC, A^{-\g}]$.

Finally, the function $g(z):=(1-z)/(1+z)^{\g+1}$ for $z \in \D$  satisfies $ \| g \vp \|_ {-(\g+1)} \le 1 $. i.e., $g \in  [\sC , A^{-\g}]$.  However,
$g \notin A^{-\g}$ since
$$
\sup_{z \in \D} \frac{(1-|z|)^{\g}|1-z|}{|1+z|^{\g+1}} \geq \sup_{s \in [0,1]} \frac{1+s}{1-s} = \infty.
$$
This shows that the containment $A^{-\g} \su [\sC, A^{-\g}]$ is proper.
\end{proof}

\begin{theorem} \label{largerextsmall}
Let $\g > 0$  and $\vp (z) := 1 / (1-z)$ for $z \in \D$.  The optimal domain $[\sC, A_0^{-\g}]$
of $\sC_{\g, 0}: A_0^{-\g} \rightarrow A_0^{-\g}$
is isometrically isomorphic to $A_0^{-\g}$ and is given by
\begin{equation}\label{eq.3.3}
[\sC, A^{-\g}_0]  =  \big  \{ f \in H(\D) \ : \ f \varphi \in A_0^{-(\g+1)} \big\} .
\end{equation}
 Moreover, the norm $\|\cdot \|_{[\sC, A_0^{-\g}]}$ is equivalent to the norm $f \mapsto \|f \vp\|_{-(\g+1)}$ and the
 containment    $A_0^{-\g} \su [\sC, A_0^{-\g}]$    is proper.
\end{theorem}

\begin{proof}
The Banach space $[\sC, A_0^{-\g}]$ is isometrically isomorphic to $A_0^{-\g}$;
adapt the  argument from  the beginning of the proof of Theorem \ref{largerextbig}.

To see  that $A_0^{-\g} \su [\sC, A_0^{-\g}]$ is proper  assume, on the contrary,
that $A_0^{-\g} = [\sC, A_0^{-\g}]$. Then  $\sC_{\g,0}:    [\sC , A_0^{-\g}] = A_0^{-\g} \rightarrow A_0^{-\g}$ is a Banach space
isomorphism and so $0 \notin \sigma(\sC_{\g,0})$;  contradiction to  Theorem \ref{T.Pe-AP}(ii).

To establish \eqref{eq.3.3} is similar to the proof of \eqref{eq.hd} in
Theorem \ref{largerextbig}. It is enough to keep in mind the following two facts.

\textit{Fact 1. } Let $f \in H(\D)$ satisfy $f(0)=0$. Then $f \in A_0^{-\g}$ if and only if $f(z)/z \in A_0^{-\g}$.

\textit{Fact 2. } A function $f \in H(\D)$ belongs to $A_0^{-\g}$ if and only if $f' \in A_0^{- (\g+1)}$.

Indeed, by  \cite[Theorem 5.5]{Du} the differentiation operator $D: A^{-\g} \rightarrow A^{-(\g+1)}$, given by
 $D(f):=f',$ and the integration operator $J: A^{- (\g+1)}  \rightarrow A^{-\g}$, given by
 $ J(f)(z):= \int_0^z f(\zeta) d \zeta$, for $ z \in \D $, are continuous (see also \cite[Theorem 2.1(a) \& Proposition 2.2(a)]{HL}).
 Since the space of polynomials $\mathcal{P}$  is dense in $A_0^{-\beta}$ for each $\beta >0$ and $\mathcal{P}$
  is invariant for both $D$ and $J$, it follows that their restrictions $D: A_0^{-\g} \rightarrow A_0^{-(\g+1)}$
  and $J: A_0^{-(\g+1)}  \rightarrow A_0^{-\g}$ are continuous. This implies Fact 2.

Now that \eqref{eq.3.3} is established, one can argue as in the proof of Theorem~\ref{largerextbig} to show that the stated norms are equivalent.
\end{proof}

\begin{remark}{\rm
(i) As noted in Section 1, for every $\g > 0$ the Banach space $A^{-\g}$ (resp. $A^{-\g}_0$) is isomorphic to
$\ell^\infty $ (resp.\ $c_0$). In particular,  $A^{-\g}$ is non-separable whereas $A^{-\g}_0$ is separable. Also,
$A^{-\g}_0$ is not weakly sequentially complete and hence, neither is $A^{-\g}$ (as it contains $A^{-\g}_0$ as a closed
subspace). Since both $c_0$ and $\ell^\infty $ fail to have the Radon-Nikodym property, the same is true for
$A^{-\g}$ and $A^{-\g}_0$. And so on. Since $A^{-\g}$ (resp. $A^{-\g}_0$) is isomorphic to $[\sC, A^{-\g}]$ (resp. to
$[\sC, A^{-\g}_0]$), we can conclude that the optimal domain spaces $[ \sC, A^{-\g}]$ and $[\sC, A^{-\g}_0]$ inherit
such properties as those mentioned above (and others) from $A^{-\g}$ and $A^{-\g}_0$, respectively.

(ii) Let $f, g \in H (\D)$ satisfy $|g (z)| \le |f (z)|$ for $z \in \D$. If $f \in A^{-\g}$ (resp. $f \in A^{-\g}_0$), for
$\g > 0  $, then it is routine to check that also $g \in A^{-\g}$ (resp. $g \in A^{-\g}_0$). It follows from \eqref{eq.hd}
(resp. \eqref{eq.3.3}) that this useful property carries over to the optimal domain space $[\sC, A^{-\g}]$
(resp. $[\sC, A^{-\g}_0]$).

(iii) According to \cite[p.92]{Du} there exists $g \in H^\infty (\D) $ such that $g'$ fails to have boundary values at
a.e. point in $\partial \D$. Since $H^\infty (\D) \su A^{-\g}_0$, for every $\g > 0$, and the differentiation operator
$D: A^{-\g}_0  \to A^{-(\g+1)}_0$ is continuous, it follows that $g' \in A^{-(\g+1)}_0$. Accordingly, for every $\beta > 1$
there exists a function in $A^{-\beta}_0$ which fails to have a.e. boundary values. Since $A^{-\beta}_0 \su
A^{-\beta} \su [\sC, A^{-\beta}]$ and $A^{-\beta}_0 \su [\sC, A^{-\beta}_0]$  we can conclude, for every $\beta > 1$,
that there exist functions in the spaces $A^{-\beta}_0, A^{-\beta}, [\sC, A^{-\beta}_0]$ and $[\sC, A^{-\beta}]$ which
fail to have a.e. boundary values. \Bo
}
\end{remark}

For $\g > 0 $, the following result shows that the largest growth space $A^{-\beta}$ that is contained in
$[\sC, A^{-\g}]$ is $A^{-\g}$. The same is true for $A^{-\g}_0$ and $[\sC, A^{-\g}_0]$.
\begin{prop}\label{p34}
Let $\g > 0 $.
For each $\beta > \g$, the space $A^{-\beta} \nsubseteq [\sC, A^{-\g}]$ and the space  $A_0^{-\beta} \nsubseteq [\sC, A_0^{-\g}]$.
\end{prop}
\begin{proof}
That $A^{-\beta} \nsubseteq [\sC, A^{-\g}]$ is a direct consequence of Theorem \ref{largerextbig}
and the fact that the function $f(z):=1/(1-z)^{\beta} \in A^{-\beta} \setminus [\sC, A^{-\g}]$ because $f(z)/(1-z) \notin A^{-(\g+1)}$.

Suppose that $A_0^{-\beta} \su [\sC, A_0^{-\g}]$. Then $\sC ( A_0^{-\beta }) \su A^{-\g}_0$. Using that $A^{-\g}_0 \su A_0^{-\beta}$
continuously and a closed graph argument it follows that the natural inclusion map $\Phi $ for  $A^{-\beta}_0 \su [\sC, A^{-\g}_0]$ is continuous.
Hence, also  $\sC : A_0^{-\beta}  \to A_0^{-\g}$ is continuous (being the composition of $\Phi $ and $\sC :  [\sC , A_0^{-\g}] \to  A_0^{-\g}$).
Passing to the bidual yields that $\sC : A^{-\beta} \to A^{-\g} $ is continuous,
i.e., $A^{-\beta} \su [\sC, A^{-\g} ]$, which is a contradiction to the previous paragraph. So, $A^{-\beta}_0 \nsubseteq [\sC, A^{-\g}_0]$.
\end{proof}

Given $\g > 0 $, Proposition \ref{p34} raises the question of whether there are any weighted Banach spaces of analytic
functions $X$ on $\D$ which satisfy $A^{-\g} \su X \su [\sC, A^{-\g}]$.

A weight $v: \D \rightarrow [0,\infty)$ is called \textit{essential} if it is radial  (i.e., $v(z)=v(|z|)$ for  $z \in \D$),
continuous, satisfies that $v(r)$ decreases to $0$ as $r$ approaches $1$, and  there exists a constant  $d>0$ such that for
 each $z_0 \in \D$ there is $f_0 \in H(\D)$ such that $|f_0(z_0)| \ge d / v (z_0)$ and $v(z)|f_0(z)| \leq 1$ for each $z \in \D$
(i.e., $\|f_0\|_v \le 1 $ in the notation of \eqref{eq.3.4} below).  It is established in \cite{BDL} that a radial continuous weight
$v : \D \to (0, \infty )$ which is decreasing to $0$ in $[0,1)$ as $r$ approaches $1$ is essential if and only if there exists a
weight $w : \D \to (0, \infty)$ satisfying the same assumptions as for $v$ such that the function $ r \mapsto  - \log (w (e^r))$ is
convex and, for some constant $c>1$, the weight $w$ satisfies
$$
\frac 1 c \, w (r) \le v (r) \le c w (r) , \qquad r \in [0,1)  .
$$
In other words, radial continuous weights in $\D$ that are essential are those which are equivalent to ``logarithmic convex weights''.
All the weights $v_{\g}(z)=(1-|z|)^{\g}$, for $\g > 0$, considered  in this paper are essential.
 For further examples and properties of essential weights we refer  to  \cite{BBT}.

Given an essential weight $v$, the \textit{weighted  space associated with $v$} is defined by
\begin{equation}\label{eq.3.4}
    H_v^\infty:= \{f \in H(\D) \ | \ \|f\|_v:=\sup_{z \in \D} v(z) |f(z)| < \infty \};
\end{equation}
it  is a Banach space of analytic functions on $\D$ for the norm $\|.\|_v$. According to
\cite[Lemma 1]{SW}, point evaluations on $\D$ belong to $(H^\infty_v)^*$.

\begin{prop}
Let $\g > 0 $ and  $v$ be an essential weight on $\D$. If it is the case that $A^{-\g} \su H_v^\infty \su [\sC,A^{-\g}]$,
then there exists $b>0$ such that
\begin{equation}\label{eq.3.5}
   (1/b)(1-|z|)^{\g} \leq v(z) \leq b (1-|z|)^{\g}   , \qquad z \in \D .
\end{equation}
In particular,  $A^{-\g} = H_v^\infty$, both  algebraically and topologically.
\end{prop}

\begin{proof}
The assumption that  $A^{-\g} \su H_v^\infty \su [\sC,A^{-\g}]$ and the closed graph theorem imply that both inclusions are continuous.
In particular,   there is $M>0$ such that $\|f\|_v \leq M \|f\|_{-\g}$ for each $f \in A^{-\g}$ and, by Theorem \ref{largerextbig},
 $\|g/(1-z)\|_{-(\g+1)} \leq K \|g\|_v$ for each $g \in H_v^\infty$ (and some $K > 0 $).

Given an arbitrary point $u \in \D \setminus \{0\}$, we set $f_0(z): =1/(1- z|u|u^{-1})^{\g}$, for $ z \in \D$.
For $u=0$, define   $f_0(z)=1$, for $ z \in \D$. Then $\|f_0\|_{-\g}=1$ and $f_0(u)=1/(1-|u|)^{\g}$.
Moreover,  $v(u)|f_0(u)| \leq \|f_0\|_v \leq M$. This implies that  $v(u) \leq M (1-|u|)^{\g}$ for each $u \in \D$.

Now fix $w \in \D$ arbitrary. Since the weight $v$ is essential, there is $d>0$ (independent of $w$) and
  $g_0 \in H(\D)$ with $\|g_0\|_v \leq 1$ (i.e., $g_0 \in H^\infty_v$) such that  $|g_0(w)| \geq d/v(w)$. Observe that
$$
\|g_0 / (1-z)  \|_{- (\g+1)} = \sup_{z \in \D}(1-|z|)^{\g+1} \frac{|g_0(z)|}{|1-z|} \leq K \|g_0\|_v \leq K.
$$
Accordingly,
$$
\frac{d}{v(w)} \frac{(1-|w|)^{\g+1}}{|1- w|}   \le (1- |w|)^{\g +1 }   \frac{|g_0 (w)|}{|1-w|} \leq
\sup_{z \in \D}(1-|z|)^{\g+1} \frac{|g_0(z)|}{|1-z|} \leq K
$$
from which it follows that
$$
\frac{(1-|w|)^{\g+1}}{|1- w|} \leq \frac{K}{d} v(w) , \qquad w \in \D .
$$
 In particular, $(1-r)^{\g} \leq (K/d) v(r)$ for  $r \in [0,1)$.
Setting $b:=\max \{M,K/d  \}$, yields \eqref{eq.3.5}.
\end{proof}

\begin{prop}\label{p36}
 Let  $\g> 0$.

{\rm(i)} The containment $A^{-\g} \su [\sC, A^{-\g}]$  is proper.

{\rm(ii)}
$ [\sC, A^{-\g}_0]$ is a proper, closed subspace of $ [\sC, A^{-\g}]$.

{\rm(iii)} The spaces $A^{-\g}$ and $[\sC, A^{-\g}_0]$ are non-comparable. That is,
$$
A^{-\g} \nsubseteq [\sC, A^{-\g}_0] \quad  \mbox{ and  } \quad  [\sC, A^{-\g}_0] \nsubseteq A^{-\g}   .
$$
\end{prop}

\begin{proof}
(i) Theorem \ref{largerextbig} shows that $A^{-\g}$ is \textit{properly} contained in $[\sC, A^{-\g}]$.

(ii) Since
$A^{-\g}_0 \su A^{-\g}$, it is clear that $[\sC, A^{-\g}_0] \su [\sC, A^{-\g}]$.    Moreover,
$ [\sC, A_0^{-\g}]$ has the same norm as that in $ [\sC, A^{-\g}]$; see \eqref{eq.norm1}. Since
$A_0^{- \g}$ is a closed subspace of $A^{- \g}$, it is routine to verify that $ [\sC, A_0^{-\g}]$ is
closed in $ [\sC, A^{-\g}]$.
It was noted in the proof of
Theorem~\ref{largerextbig} that the function $g (z):= (1-z) / (1+z)^{\g+1}$ for $z \in \D$ belongs to
$[\sC, A^{-\g}]$. However,
$$
(1-|z|)^{\g+1} \Big|\frac{g (z)}{1-z}\Big| = \Big (\frac{1-|z|}{|1+z|}\Big)^{\g+1} , \qquad z \in \D ,
$$
has the  value 1 for every real $z \in (-1,0]$ and so, via \eqref{eq.3.3}, $g \notin [\sC, A^{-\g}_0]$.
Hence, the containment $ [\sC, A_0^{-\g}] \su  [\sC, A^{-\g}]$ is \textit{proper}.

(iii) We first show that $A^{-\g} \nsubseteq [\sC, A_0^{-\g}]$. On the contrary, if $A^{-\g} \su [\sC, A_0^{-\g}]$, then it follows that
$ \sC ( A^{-\g}) \su  A_0^{-\g}$. Since $( A_0^{-\g})^{**} = A^{-\g}  $ and
$(\sC_{\g,0})^{**} = \sC_\g$, we can conclude that $(\sC_{\g,0})^{**} (( A_0^{-\g})^{**}) \su  A_0^{-\g}$. But, this
contradicts the fact that $\sC_{\g, 0}$ is not weakly compact; see Corollary 2.2 above and \cite[Ch.\,VI, Sect.\ 4, Theorem 2]{DSI}.
Accordingly, $ A^{-\g} \nsubseteq [\sC,  A_0^{-\g}]$.

To show that $[\sC, A_0^{-\g}] \nsubseteq A^{-\g}$ assume, on the contrary, that $  [\sC, A_0^{-\g}] \su A^{-\g}$. Via part (i), to
achieve a contradiction it suffices to deduce that $[\sC, A^{-\g}] \su A^{-\g}$.

By the closed graph theorem the inclusion $ [\sC, A_0^{-\g}] \su A^{-\g}$ is continuous. Hence, there exists $M>0$ such that
$$
\|h\|_{-\g} \le M \|h\|_{ [\sC, A^{-\g}]}  , \qquad h \in [\sC, A_0^{-\g}],
$$
after noting that the spaces $ [\sC, A_0^{-\g}]$ and $ [\sC, A^{-\g}]$ have the same norm.  Fix  $f \in  [\sC, A^{-\g}]$, in which case
$\sC (f) \in A^{-\g}$. According to Proposition 1.2 of \cite{BBG} there exists a sequence of polynomials $\{p_j\}_{j \in \N }
\su A^{-\g}_0 \su [\sC, A_0^{-\g}]  $  which converges to $\sC (f)$ in $H (\D)$ and satisfies $\|p_j\|_{-\g} \le \|\sC (f)\|_{-\g}$
for $j \in \N $. Then $\{\sC^{-1} (p_j)\}_{j \in \N }$ converges to $f$ in $H (\D)$ with $\{\sC^{-1} (p_j)\}_{j \in \N } \su [\sC, A_0^{-\g}]
\su A^{-\g}$ (where the last containment is being assumed). Moreover, $\{\sC^{-1} (p_j)\}_{j \in \N }$ is a bounded set in $A^{-\g}$
because
$$
\|\sC^{-1} (p_j)\|_{- \g} \le M \|\sC^{-1} (p_j)\|_{ [\sC, A^{-\g}]} = M \|p_j\|_{- \g} \le M \|C (f)\|_{- \g} , \qquad j \in \N .
$$
Since closed balls of $A^{-\g}$ are compact sets in the Montel space $H (\D)$, it follows that $f \in A^{-\g}$. This shows that
$ [\sC, A^{-\g}] \su A^{-\g} $ which yields the required contradiction.
\end{proof}

Proposition \ref{p36} shows that the Banach spaces of analytic functions $A^{-\g}$ and $[\sC, A_0^{-\g}]  $ are two
\textit{non-comparable, proper} vector subspaces of the optimal domain space $[\sC, A^{-\g}]$, that is, $\sC$  maps
both of these spaces into $ A^{-\g}$.

\bigskip

\textbf{Acknowledgements.}
The research of the first two authors was partially
supported by the projects  MTM2013-43540-P and GVA Prometeo II/2013/013.

\bigskip

\end{document}